\documentclass[oneside,english,reqno]{amsart}
\usepackage[T1]{fontenc}
\usepackage[latin9]{inputenc}
\usepackage{babel}
\usepackage{amsthm}
\usepackage{amstext}
\usepackage{amssymb}
\usepackage[unicode=true,pdfusetitle,
 bookmarks=true,bookmarksnumbered=false,bookmarksopen=false,
 breaklinks=false,pdfborder={0 0 1},backref=false,colorlinks=false]
 {hyperref}
\hypersetup{
 colorlinks=true,citecolor=blue,linkcolor=blue,linktocpage=true}

\makeatletter
\numberwithin{equation}{section}
\numberwithin{figure}{section}
\theoremstyle{plain}
\newtheorem{thm}{\protect\theoremname}[section]
  \theoremstyle{remark}
  \newtheorem{rem}[thm]{\protect\remarkname}
  \theoremstyle{definition}
  \newtheorem{example}[thm]{\protect\examplename}
  \theoremstyle{plain}
  \newtheorem{lem}[thm]{\protect\lemmaname}
  \theoremstyle{plain}
  \newtheorem{cor}[thm]{\protect\corollaryname}

\renewcommand{\section}{%
\@startsection{section}{1}%
  \z@{.7\linespacing\@plus\linespacing}{.5\linespacing}%
  {\normalfont\scshape\centering\bfseries}}

\renewcommand{\subsection}{%
\@startsection{subsection}{2}%
  \z@{.5\linespacing\@plus.7\linespacing}{.5\linespacing}%
  {\normalfont\bfseries}}

\allowdisplaybreaks

\makeatother

  \providecommand{\corollaryname}{Corollary}
  \providecommand{\examplename}{Example}
  \providecommand{\lemmaname}{Lemma}
  \providecommand{\remarkname}{Remark}
\providecommand{\theoremname}{Theorem}

\begin{document}
\subjclass[2010]{ 11K55, 28A78, 28A80, 51F99. }

\title[Hausdorff Measure of Cantor Sets]{Exact Hausdorff Measure of Certain Non-Self-Similar Cantor Sets}\footnote{The final version of this paper will appear in Fractals.}

\author{Steen Pedersen and Jason D. Phillips}

\address{Department of Mathematics, Wright State University, Dayton OH 45435.}

\email{steen@math.wright.edu}

\email{phillips.50@wright.edu}

\keywords{Cantor set, fractal set, self-similar set, Hausdorff measure, density. }
\begin{abstract}
We establish a formula yielding the Hausdorff measure for a class
of non-self-similar Cantor sets in terms of the canonical covers of
the Cantor set. 
\end{abstract}
\maketitle
\tableofcontents{}

\section{Introduction}

In this paper we study the Hausdorff measure of certain non-self-similar
Cantor sets. The simplest form of our results is Theorem \ref{thm:Main}.
This theorem contains a formula for the Hausdorff measure of certain
Cantor sets. There are related results in the literature, for different
classes of set, e.g, self-similar Cantor sets Marion \cite{Ma86,Ma87},
Ayer and Strichartz \cite{AySt99}, homogeneous Cantor sets Qu, Rao,
and Su \cite{QRS01}, and intersections of integral self-affine sets
Bondarenko and Kravchenko \cite{BoKr11}.  The class of Cantor sets
we consider overlaps each of these three papers, but it contains Cantor
sets that are not covered by any of these papers, in fact, Example
\ref{Sec-3:Example-3.1} contains an explicit construction of such
Cantor sets.

Estimates of the Hausdorff measure of various classes of linear Cantor
sets can, for example, be found in Cabrelli, Mendivil, Molter, and
Shonkwiler \cite{CMMS04}, Feng, Rao, and Wu \cite{FRJ96}, Garcia,
Molter, and Scotto \cite{GMS07}, Hare, Mendivil, and Zuberman \cite{HMZ13},
Marion \cite{Ma87}, and Pedersen and Phillips \cite{PePh13}. 

While Cantor sets may appear to be special, they occur in mathematical
models involving fractals (e.g., iterated function systems and self-similar
measures); they play a role in number theory (e.g., in $b$-ary number
representations, where $b$ is a base for a number system); in signal
processing and in ergodic theory (e.g., in development of codes as
beta-expansions); and in limit-theorems from probability (e.g., explicit
properties of infinite Bernoulli convolutions.) Cantor sets are also
rooted in the theory of dynamical systems. See, for instance, Palis
and Takens \cite{PaTa87}. For early research in this area see Davis
and Hu \cite{DaHu95} and the list of references therein. In addition
to the papers cited above, we list a small sample, of the numerous
papers dealing with the subject, neighboring areas and applications:
Furstenberg, \cite{MR0354562}, Williams \cite{MR1112813}, Kraft
\cite{MR1106988}, \cite{MR1653359}, Peres and Solomyak \cite{MR1491873},
Duan, Liu, and Tang \cite{MR2534911}, Moreira \cite{MR2810854}. 

We refer to \cite{Fal85} for background information on Hausdorff
measures and Cantor sets.

\section{Statement of Results}

\subsection{\label{sub:Interval-Construction}Interval Construction}

The Cantor sets we consider are determined by a refinement process.
This is a generalization of the familiar interval construction of
the middle thirds Cantor set. The refinements we consider are determined
by parameters $m\geq2,$ $0<\beta<\tfrac{1}{2},$ and $g_{j},$ $j=0,1,\ldots,m,$
such that $g_{0}\geq0,$ $g_{j}>0$ for $j=1,\ldots,m-1,$ $g_{m}\geq0,$
and 
\[
m\beta+\sum_{j=0}^{m}g_{j}=1.
\]
For a closed interval $I,$ the refinement of $I$ determined by $m,$
$\beta,$ and $g_{j}$ is the collection of $m$ closed intervals
$I_{j},$ $j=1,\ldots,m,$ where $I_{j}$ is to the left of $I_{j+1},$
the intervals $I_{j}$ all have length $\left|I_{j}\right|=\beta\left|I\right|,$
the gap between the left hand endpoint of $I$ and $I_{1}$ has length
$g_{0}\left|I\right|,$ the gap between $I_{j}$ and $I_{j+1}$ has
length $g_{j}\left|I\right|,$ and the gap between $I_{m}$ and the
right hand endpoint of $I$ has length $g_{m}\left|I\right|.$ Writing
$I=\left[x,y\right]$ and $I_{j}=\left[x_{j},y_{j}\right]$ we can
state these conditions as $y_{j}\leq x_{j+1},$ $y_{j}-x_{j}=\beta\left(y-x\right),$
$x_{1}-x=g_{0}\left(y-x\right),$ $x_{j+1}-y_{j}=g_{j}\left(y-x\right),$
and $y-y_{m}=g_{m}\left(y-x\right).$ In particular, we do not assume
that the gaps between the retained intervals have the same lengths
or that the endpoints of $I$ are in $\bigcup I_{j}.$ 

The class of Cantor sets $C$ we consider is obtained by selecting
a sequence $\beta_{k},$ $m_{k},$ $g_{k,j},$ $j=0,\ldots,m_{k},$
$k=1,2,\ldots$ of refinement parameters. The Cantor set $C$ is then
constructed by induction. Let $C_{0}:=\left[0,1\right]$ and construct
$C_{k}$ from $C_{k-1}$ by applying the same refinement process determined
by $\beta_{k},$ $m_{k},$ $g_{k,j}$ to each interval in $C_{k-1}.$
The Cantor set $C$ is the intersection $C:=\bigcap C_{k}.$ 

Imposing appropriate conditions on the parameters $\beta_{k},$ $m_{k},$
$g_{k,j},$ $j=0,\ldots,m_{k},$ $k=1,2,\ldots$ allows us to calculate
the Hausdorff measure of $C.$ 
\begin{rem}
The middle thirds Cantor set is obtained by setting $m_{k}=2,$ $\beta_{k}=\tfrac{1}{3},$
$g_{k,0}=g_{k,2}=0,$ and $g_{k,1}=\tfrac{1}{3}$ for all $k=1,2,\ldots.$ 
\end{rem}

\subsection{Function System Construction}

We rewrite the construction above in a way that is more convenient
for our purposes. Essentially, we give a construction of the left
hand endpoints of the intervals in $C_{k}.$ 

Fix a sequence $0<\beta_{j}<1/2.$ Then $\frac{1-\beta_{j}}{\beta_{j}}>1.$
For each $j\in\mathbb{N},$ let $D_{j}$ be a finite subset of the
closed interval $\left[0,\frac{1-\beta_{j}}{\beta_{j}}\right]$ containing
at least two elements satisfying $\left|d-d'\right|>1$ for all $d\neq d'$
in $D_{j}.$ Since $D_{j}$ is a subset of $\left[0,\frac{1-\beta_{j}}{\beta_{j}}\right]$,
the maps

\[
f_{j,d}(x):=\beta_{j}\left(x+d\right),x\in\mathbb{R},
\]
map the closed interval $\left[0,1\right]$ into itself, for all $j\in\mathbb{N}$
and all $d\in D_{j}.$ For a set $A,$ let 
\[
f_{j}\left(A\right):=\bigcup_{d\in D_{j}}f_{j,d}\left(A\right).
\]
Let $C_{0}:=[0,1],$ $C_{1}:=f_{1}\left(\left[0,1\right]\right),$
$C_{2}:=f_{1}\left(f_{2}\left(\left[0,1\right]\right)\right),$ etc.
Since $f_{k}\left(\left[0,1\right]\right)\subseteq\left[0,1\right]$
it follows that $C_{k}\subseteq C_{k-1}$ for $k\geq1.$  Let $b(0):=1,$
and 
\begin{equation}
b(k):=\beta_{1}\beta_{2}\cdots\beta_{k}\label{Sec-2-eq:b-Def}
\end{equation}
for $k\geq1.$ Clearly, 
\begin{eqnarray*}
C_{k} & = & \left\{ \sum_{j=1}^{k}d_{j}b(j)\mid d_{j}\in D_{j}\right\} +\left[0,b(k)\right]\\
 & = & \bigcup_{d_{1}\in D_{1},\ldots,d_{k}\in D_{k}}\left[\sum_{j=1}^{k}d_{j}b(j),b(k)+\sum_{j=1}^{k}d_{j}b(j)\right],
\end{eqnarray*}
for $k\geq1.$  The union is disjoint, since $\left|d_{j}-d_{j}'\right|>1$
for all $d_{j}\neq d_{j}'$ in $D_{j}.$ More precisely, if $m_{k}$
is the number of elements of $D_{k}$ and 
\[
D_{k}=\left\{ d_{k,j}\mid j=0,\ldots,m_{k}-1\right\} ,
\]
where $d_{k,j}<d_{k,j+1}$ for $j=0,1,\ldots,m_{k}-2$, then, in the
notation of Section \ref{sub:Interval-Construction}, $g_{k,0}=\beta_{k}d_{k,0},$
$g_{k,m_{k}}=1-\beta_{k}\left(d_{k,m_{k}-1}+1\right),$ and 
\[
g_{k,j}:=\beta_{k}\left(d_{k,j}-\left(d_{k,j-1}+1\right)\right),
\]
$j=1,\ldots,m_{k}-1.$ See also, Remark \ref{Rem:distance-between-intervals}.

Let 
\begin{equation}
C=C{}_{\beta_{1},\beta_{2},\ldots,D_{1},D_{2},\ldots}:=\bigcap_{k=0}^{\infty}C_{k}.\label{eq:C-definition}
\end{equation}

A \emph{basic interval of order $k$} is a closed interval of the
form 
\[
f_{1,d_{1}}\circ\cdots\circ f_{k,d_{k}}\left(\left[0,1\right]\right)=\sum_{j=1}^{k}d_{j}b(j)+\left[0,b(k)\right]=\left[\sum_{j=1}^{k}d_{j}b(j),b(k)+\sum_{j=1}^{k}d_{j}b(j)\right],
\]
where $d_{j}$ is in $D_{j},$ with the understanding that the interval
$C_{0}=[0,1]$ is the only basic interval of order $0.$ Note, $C_{k}$
is the union of the $\mu(k)$ basic intervals of order $k,$ where
\begin{equation}
\mu(k)=\mu_{D_{1},D_{2},\ldots}(k):=m_{1}m_{2}\cdots m_{k}.\label{eq:mu-Def}
\end{equation}
Since $C_{k+1}\subseteq C_{k}$ each basic interval of order $k+1$
is contained in some basic interval of order $k.$ In fact, each basic
interval of order $k$ contains $m_{k+1}$ basic intervals of order
$k+1.$ A \emph{simple interval of order $k$} is the convex hull
of two, not necessarily consecutive, basic intervals of order $k$\emph{.} 
\begin{rem}
\label{Rem:distance-between-intervals}If $x'=\sum_{j=1}^{k}d_{j,i_{j}'}b(j)$
and $x=\sum_{j=1}^{k}d_{j,i_{j}}b(j)$ are the left endpoints of consecutive
intervals in $C_{k},$ then there is an $1\leq\ell\leq k,$ such that
$i_{j}=i_{j}'$ for $j<\ell,$ $i_{\ell}=i_{\ell}'+1,$ and $i_{j}=0,$
$i_{j}'=m_{j}-1$ for $\ell<j\leq k.$ Hence, the gap between these
two intervals is 
\begin{align*}
g & =x-x'-b(k)\\
 & =\left(d_{\ell,i_{\ell}}-d_{\ell,i_{\ell}-1}\right)b\left(\ell\right)-\sum_{j=\ell+1}^{k}\left(d_{j,m_{j}-1}-d_{j,0}\right)b\left(j\right)-b\left(k\right)\\
 & \geq\left(d_{\ell,i_{\ell}}-d_{\ell,i_{\ell}-1}\right)b\left(\ell\right)-\sum_{j=\ell+1}^{k}\frac{1-\beta_{j}}{\beta_{j}}b\left(j\right)-b\left(k\right)\\
 & =\left(d_{\ell,i_{\ell}}-\left(d_{\ell,i_{\ell}-1}+1\right)\right)b\left(\ell\right)>0.
\end{align*}
A similar calculation is used, in the proof of Lemma \ref{lem:Positive-alphas-only},
to compare the lengths of certain simple intervals. 
\end{rem}

\subsection{Hausdorff Measure of Linear Cantor Sets}

If $\beta_{k}\geq\beta>0$ for all $k,$ then it is known, see e.g.,
\cite{Ma86,Ma87}, that the Hausdorff dimension of $C$ is 
\begin{equation}
s=s_{\beta_{1},\beta_{2},\ldots,D_{1},D_{2},\ldots}:=\liminf_{k\to\infty}\frac{\log\left(\mu(k)\right)}{-\log(b(k))}=\liminf_{k\to\infty}\frac{\log\left(m_{1}m_{2}\cdots m_{k}\right)}{-\log\left(\beta_{1}\beta_{2}\cdots\beta_{k}\right)}.\label{eq:Haudorff-dim-formula}
\end{equation}
In particular, the Hausdorff dimension of $C$ is determined solely
by the length and number of basic intervals at each stage. The Hausdorff
dimension of $C$ can depend on the $D_{j},$ in some surprising ways,
see e.g. \cite{PePh11}. A special case of Theorem \ref{Sec-7-thm:Main-Complete}
is 
\begin{thm}[Measure Formula Theorem]
\label{thm:Main} Let $0<\beta_{k}<1/2$ and
\begin{equation}
D_{k}=\left\{ d_{k,j}\mid j=0,\ldots,m_{k}-1\right\} \subset\left[0,\frac{1-\beta_{k}}{\beta_{k}}\right],\label{Sec-2-eq:Def-of-D}
\end{equation}
where 
\begin{align*}
\max\left\{ 2,\left(1+j-i\right)^{1/s}-1\right\}  & \leq d_{k,j}-d_{k,i}
\end{align*}
for all $0\leq i<j<m_{k}$ and all $k.$ Let $s$ be determined by
(\ref{eq:Haudorff-dim-formula}) and let $C$ be as in (\ref{eq:C-definition}).
Suppose $m_{k}\beta_{k}^{s}\leq1$ for all $k$, then the $s$--dimensional
Hausdorff measure of $C$ equals
\begin{equation}
L:=\liminf_{k\to\infty}\mu(k)b(k)^{s},\label{Sec-2-eq:Def-of-L}
\end{equation}
where $\mu(k)=m_{1}\cdots m_{k}$ and $b(k)=\beta_{1}\cdots\beta_{k}$
are as in (\ref{eq:mu-Def}) and (\ref{Sec-2-eq:b-Def}). \end{thm}
\begin{rem}
It follows from the assumptions in Theorem \ref{thm:Main} that both
limit inferiors in (\ref{eq:Haudorff-dim-formula}) and (\ref{Sec-2-eq:Def-of-L})
are limits. In fact, $s$ is the supremum of the set $\left\{ \frac{\log\mu(k)}{-\log b(k)}\mid k\geq1\right\} $
and clearly $L\leq1$ and
\[
L=\inf\left\{ \left.\frac{\log\mu(k)}{-\log b(k)}\right|k\geq1\right\} .
\]
In particular, if $\sum_{k=1}^{\infty}\left(1-m_{k}\beta_{k}^{s}\right)$
is convergent, then $C$ is an $s$--set. See Section \ref{sec-6:Examples}
for details. 
\end{rem}
It was shown in \cite{QRS01} that (\ref{Sec-2-eq:Def-of-L}) determines
the Hausdorff measure of a class of homogeneous Cantor sers. In the
notation of Section \ref{sub:Interval-Construction} \emph{homogeneous
Cantor sets} are detemined by the conditions $g_{k,0}=g_{k,m_{k}}=0$
and $g_{k,j}=g_{k,1},$ for $j=1,\ldots,m_{k}-1$ and all $k.$ We
state the result from \cite{QRS01} in the notation of Theorem \ref{thm:Main}:
\begin{thm}[\cite{QRS01}]
 \label{Sec-2-thm:-QRS} Let $0<\beta_{k}<1/2$ and $m_{k}\geq2$
with $\beta_{k}m_{k}<1$ be given. Set $d_{k}:=\frac{1-\beta_{k}}{\beta_{k}\left(m_{k}-1\right)}$
and $d_{k,j}=jd_{k},$ $j=0,\ldots,m_{k}-1,$ Let $s$ be determined
by (\ref{eq:Haudorff-dim-formula}), let $C$ be as in (\ref{eq:C-definition}),
and let $L$ be determined by (\ref{Sec-2-eq:Def-of-L}). If $\beta_{k+1}\left(d_{k+1}-1\right)\leq d_{k}-1,$
then $L$ is the $s$--dimensional Hausdorff measure of $C$. 
\end{thm}
In Remark \ref{Sec-6-Rem:QRS} we construct examples satisfying the
assumptions of Theorem \ref{thm:Main}, but not the assumptions of
Theorem \ref{Sec-2-thm:-QRS} and examples satisfying the assumptions
of Theorem \ref{Sec-2-thm:-QRS}, but not the assumptions of Theorem
\ref{thm:Main}. 

For self-similar set, (i.e. $\beta_{k}=\beta$ and $D_{k}=D$ for
all $k,$) satisfying an open set condition there is a different approach
to the calculation of the Hausdorff measure of $C$ in \cite{Ma86,Ma87}
and \cite{AySt99}.

\section{The Measure Theorem}

Equation (\ref{Sec-2-eq:Def-of-L}) states that we can find the Hausdorff
measure, by only considering covers by basic intervals at stage $k.$
So we need cover by basic interval to be more ``efficient'' than
covers by simple intervals. Hence, if $P$ is a simple interval at
stage $k$ containing the basis intervals $I_{j},$ then we need $|P|^{s}\geq\sum_{j}|I_{j}|^{s}.$
It turns out that this is a separation condition on the basic intervals. 
\begin{example}
\label{Sec-3:Example-3.1}We will illustrate the nature of this separation
condition in the special case where $\beta_{k}=\beta$ and $D_{k}=\left\{ 0,d_{k,1},n-1\right\} ,$
here $n:=1/\beta.$ Then $s=\log(3)/\log(n).$ By considering the
case where $P$ is the convex hull of an adjacent pairs of basic intervals,
we see that $|P|^{s}\geq|I_{1}|^{s}+|I_{2}|^{s}$ reduces to $(d_{k,1}+1)^{s}\geq2$
and $(n-d_{k,1})^{s}\geq2.$ We can write these conditions as
\[
\frac{\log(d_{k,1}+1)}{\log(n)}\geq\frac{\log(2)}{\log(3)}\quad\text{and}\quad\frac{\log(n-d_{k,1})}{\log(n)}\geq\frac{\log(2)}{\log(3)}.
\]
If $n=9,$ these conditions reduce to $3\leq d_{k,1}\leq5.$ If $n=81,$
these conditions reduce to $15\leq d_{k,1}\leq65.$ In particular,
the separation condition on the basic intervals must depend on $\beta.$ 
\end{example}
Let $0<\beta_{k}<1/2$ and let $D_{k}$ be as in (\ref{Sec-2-eq:Def-of-D}).
Let $t:=1/s,$ where $s$ is determined by (\ref{eq:Haudorff-dim-formula}).
Suppose $0\leq i,j<m_{k'}$ for some $k'\geq1.$ 
\begin{description}
\item [{Assumption-1}] If $i<j$ and $k'\leq k,$ then
\[
\left(\left(1+j-i\right)^{t}-1\right)\frac{b(k)}{b(k')}\left(\frac{\mu\left(k\right)}{\mu\left(k'\right)}\right)^{t}\leq d_{k',j}-d_{k',i}.
\]

\item [{Assumption-2a}] If $i<j$ and $k'<k,$ then
\[
1+\left(j-i\right)^{t}\frac{b(k)}{b(k')}\left(\frac{\mu\left(k\right)}{\mu\left(k'\right)}\right)^{t}\leq d_{k',j}-d_{k',i}.
\]

\item [{Assumption-2b}] If $i<j$, then 
\[
2\left(j-i\right)\leq d_{k',j}-d_{k',i}.
\]

\end{description}
Note, $\frac{\mu\left(k\right)}{\mu\left(k'\right)}=m_{k'+1}\cdots m_{k}$
and $\frac{b(k')}{b(k)}=\frac{1}{\beta_{k'+1}}\cdots\frac{1}{\beta_{k}}.$
So, when $k=k',$ Assump\-tion-1 states $\left(1+j-i\right)^{t}\leq1+d_{k',j}-d_{k',i}.$
Of course, Assumption-2b simply means that $2\leq d_{k',i+1}-d_{k',i}$
for all $0\leq i<m_{k'}-1.$

It is possible to have $k=k'$ in Assumption-1, but not in Assumption-2a.
This will be important in some of our examples in Section \ref{sec-6:Examples}.
The purpose of Assumptions 1 and 2 is that they allow us to show that
the most ``economical'' covers by simple intervals are the covers
by basic intervals. This is the content of the following lemma. 
\begin{lem}
\label{lem:Basic-vs-Simple-Intervals}Suppose Assumptions 1 and 2a
hold or Assumptions 1 and 2b hold. For any $k\geq1,$ and any simple
interval $P$ at stage $k,$ $|P|^{s}\geq i|I|^{s},$ where $|I|$
is the length of a basic interval at stage $k,$ $i$ is the number
of basic intervals at stage $k$ contained in $P,$ and $s$ is determined
by (\ref{eq:Haudorff-dim-formula}). 
\end{lem}
We prove Lemma \ref{lem:Basic-vs-Simple-Intervals} in Section \ref{sec:Proof-of-Lemma-Basic-vs-Simple}.
This is the key technical result in this paper. 

Let $\alpha\geq0$ be a real number. Recall, e.g., \cite{Fal85},
the $\alpha$--dimensional Hausdorff measure of a set $C$ is 
\[
H^{\alpha}(C):=\lim_{\delta\to0}H_{\delta}^{\alpha}(C),
\]
where 
\[
H_{\delta}^{\alpha}(C):=\inf\left\{ \sum\left|G_{i}\right|^{\alpha}\mid C\subseteq\bigcup G_{i},|G_{i}|<\delta,C\subset\bigcup G_{i}\right\} .
\]
The infimum is over all countable covers of $C,$ by sets $G_{i}$
satisfying $\left|G_{i}\right|<\delta.$ The infimum is not changed
if we only consider countable covers by open intervals, or only consider
countable covers by closed intervals. 

The usefulness of simple intervals is illustrated by: 
\begin{lem}
\label{lem:Haudorff-measure-simple-intervals}The $s$--dimensional
Hausdorff measure $H^{s}(C)$ of $C$ equals 
\[
\lim_{\delta\to0}\inf\left\{ \sum\left|P_{i}\right|^{s}\mid C\subseteq\bigcup P_{i},P_{i}\text{ simple interval of order }0\leq k,\left|P_{i}\right|<\delta\right\} ,
\]
where $|P|$ denotes the diameter of the set $P$ and $s$ is determined
by (\ref{eq:Haudorff-dim-formula}). 
\end{lem}
Since the number of basic intervals of order $k$ is finite, so is
the number of simple intervals of order $k.$ In particular, the sum
$\sum\left|P_{i}\right|^{s}$ is finite. A result, similar to Lemma
\ref{lem:Haudorff-measure-simple-intervals}, can be found in \cite{Ma86}.
We give a simple proof in Section \ref{sec:Proof-of-Lemma-Haudorff-measure-simple-intervals]}. 

As an immediate consequence of Lemma \ref{lem:Basic-vs-Simple-Intervals}
and Lemma \ref{lem:Haudorff-measure-simple-intervals} we have:
\begin{thm}[Measure Theorem]
\label{thm:Measure-Theorem}If Assumptions 1 and 2a hold or Assumptions
1 and 2b hold, then the $s$--dimensional Hausdorff measure of $C$
equals
\[
\lim_{\delta\to0}\inf\left\{ \sum\left|I_{i}\right|^{s}\mid C\subseteq\bigcup I_{i},I_{i}\text{ basic intervals all of the same order},\left|I_{i}\right|<\delta\right\} .
\]
Here $s$ is determined by (\ref{eq:Haudorff-dim-formula}). 
\end{thm}
For any $k,$ there is a finite number of basic intervals of order
$k,$ so the sum $\sum\left|I_{i}\right|^{s}$ is finite.

\section{Proof of Lemma \ref{lem:Haudorff-measure-simple-intervals} \label{sec:Proof-of-Lemma-Haudorff-measure-simple-intervals]}}

Lemma \ref{lem:Haudorff-measure-simple-intervals} is a direct consequence
of the following lemma.
\begin{lem}
\label{lem:Delta-Simple}For any real number $\alpha\geq0$ and any
$\delta>0$ 
\[
H_{\delta}^{\alpha}(C)=\inf\left\{ \sum\left|P_{i}\right|^{\alpha}\mid0\leq k,P_{i}\text{ simple interval of order }k,|P_{i}|<\delta,C\subset\bigcup P_{i}\right\} ,
\]
where $|Q|$ is the diameter of the set $Q$ and the infimum is over
all $k\geq0$ and over all (necessarily finite) covers of $C$ by
simple intervals $P_{i}$ of order $k$ satisfying $\left|P_{i}\right|<\delta.$ \end{lem}
\begin{rem}
The proof shows we may replace $k\geq0$ by $k\in A,$ where $A$
is any (fixed) infinite set of positive integers, and also that we
may allow simple intervals of different orders.  \end{rem}
\begin{proof}
Since covers by simple intervals are a subset of all covers we have
\begin{align*}
H_{\delta}^{\alpha}(C) & \leq\inf\left\{ \sum\left|P_{i}\right|^{\alpha}\mid P_{i}\text{ simple interval},|P_{i}|<\delta,C\subset\bigcup P_{i}\right\} \\
 & \leq\inf\left\{ \sum\left|P_{i}\right|^{\alpha}\mid k\geq0,P_{i}\text{ simple interval of order }k,|P_{i}|<\delta,C\subset\bigcup P_{i}\right\} .
\end{align*}
We establish the reverse inequality. Suppose $\bigcup G_{i}$ is a
countable open cover of $C$ such that $\left|G_{i}\right|<\delta$
for all $i.$ We must show there is a $k\geq0$ and a cover $\bigcup P_{i}$
of $C$ by simple intervals $P_{i}$ of order $k,$ such that $\left|P_{i}\right|<\delta$
and $\sum\left|P_{i}\right|^{\alpha}\leq\sum\left|G_{i}\right|^{\alpha}.$
Since $C_{N}\to C$ as $N\to\infty$ with respect to the Hausdorff
metric, there is an $N$ such that $C_{N}$ is contained in $\bigcup G_{i}.$
Let $\varepsilon>0$ be a Lebesgue number associated with the open
cover $\bigcup G_{i}$ of the compact set $C_{N}.$ Then any subset
of $C_{N}$ with diameter $<\varepsilon$ must be contained in some
$G_{i}.$ Fix $k\geq N$ such that $b\left(k\right)<\varepsilon,$
then any $n-$ary interval in $C_{k}$ is contained in some $G_{i}.$
Let $Q_{i}$ be the union of the $n-$ary intervals in $C_{k}$ that
are contained in $G_{i}.$ Then $Q_{i}\subseteq G_{i}$ and $\bigcup Q_{i}$
is a cover of $C_{k},$ consequently also of $C,$ and 
\[
\sum\left|Q_{i}\right|^{\alpha}\leq\sum\left|G_{i}\right|^{\alpha}.
\]
If $P_{i}$ is the convex hull of $Q_{i}$, then $P_{i}$ is a simple
interval with the same diameter as $Q_{i}.$ Since $Q_{i}$ is a subset
of $G_{i}$ and $G_{i}$ has diameter $<\delta$ it follows that $\bigcup P_{i}$
is a cover by simple intervals of order $k$ each of diameter $<\delta.$ 
\end{proof}

\section{Proof of Lemma \ref{lem:Basic-vs-Simple-Intervals}.\label{sec:Proof-of-Lemma-Basic-vs-Simple}}

For $\eta_{j}\in\left\{ 0,1,\ldots,m_{j-1}\right\} $ set $\phi_{j}\left(\eta_{j}\right):=d_{j,\eta_{j}}.$
At stage $k,$ the left hand endpoints of the basic intervals are
of the form 
\begin{equation}
x=\sum_{j=1}^{k}\phi_{j}\left(\varepsilon_{j}\right)b(j).\label{eq:Left-End-Point}
\end{equation}
Consequently, the right hand endpoints of the basic intervals are
of the form 
\begin{equation}
y=b(k)+\sum_{j=1}^{k}\phi_{j}\left(\eta_{j}\right)b(j).\label{eq:Right-End-Point}
\end{equation}
The number of basic intervals contained in $[0,y]$ is 
\[
i=1+\sum_{j=1}^{k}\eta_{j}\frac{\mu(k)}{\mu(j)}=1+\sum_{j=1}^{k}\eta_{j}m_{j+1}m_{j+2}\cdots m_{k}
\]
and the number of basic intervals contained in $[0,x)$ is 
\[
i=\sum_{j=1}^{k}\varepsilon_{j}\frac{\mu(k)}{\mu(j)}=\sum_{j=1}^{k}\varepsilon_{j}m_{j+1}m_{j+2}\cdots m_{k}.
\]
Let $x,y$ be as in (\ref{eq:Left-End-Point}) and (\ref{eq:Right-End-Point}).
Suppose $x<y,$ let $k'$ be the smallest subscript, if any, for which
$\eta_{k'}-\varepsilon_{k'}\neq0,$ then $\eta_{k'}-\varepsilon_{k'}>0.$
Furthermore, any simple interval at stage $k$ is of the form $P=[x,y]$
for some $x<y$ as in (\ref{eq:Left-End-Point}) and (\ref{eq:Right-End-Point}).
The length $|P|=y-x$ of the simple interval $P=[x,y]$ is 
\begin{equation}
|P|=b(k)+\sum_{j=1}^{k}\left(\phi_{j}\left(\eta_{j}\right)-\phi_{j}\left(\varepsilon_{j}\right)\right)b(j)\label{eq:Length-Of-P}
\end{equation}
and the number of basic intervals at stage $k$ contained in $P$
is 
\begin{equation}
i=i_{P}=1+\sum_{j=1}^{k}\left(\eta_{j}-\varepsilon_{j}\right)m_{j+1}m_{j+2}\cdots m_{k}.\label{eq:Number-Of-Basic-Intervals-In-P}
\end{equation}

Let $s$ be determined by (\ref{eq:Haudorff-dim-formula}). We wish
to show that 
\begin{equation}
i^{t}\leq|P|/|I|,\label{eq:Goal-P}
\end{equation}
where $t=1/s,$ $P$ is a simple interval at stage $k,$ and $I$
is a basic interval at stage $k.$ Suppose $P=[x,y],$ where $x<y$
are as in (\ref{eq:Left-End-Point}) and (\ref{eq:Right-End-Point}). 

To simplify the notation, let $\alpha_{j}:=\eta_{j}-\varepsilon_{j}$
and 
\[
g_{j}\left(\alpha\right):=\min\left\{ \phi_{j}\left(\eta\right)-\phi_{j}\left(\varepsilon\right)\mid\eta,\varepsilon\in D_{j},\alpha=\eta-\varepsilon\right\} .
\]
Using this notation (\ref{eq:Goal-P}) takes the form 
\begin{equation}
\left(1+\sum_{j=1}^{k}\alpha_{j}m_{j+1}m_{j+2}\cdots m_{k}\right)^{t}\leq1+\sum_{j=1}^{k}g_{j}\left(\alpha_{j}\right)\frac{b(j)}{b(k)}.\label{eq:Goal-Alpha}
\end{equation}

Note $1-m_{j}\leq\alpha_{j}\leq m_{j}-1$ and $1-n\leq g_{j}\left(\alpha_{j}\right)\leq n-1.$
Given $\alpha_{j,}$ by definition of $g_{j},$ there are sequences
$\eta'_{j},$ $\varepsilon'_{j}$ such that $\phi_{j}\left(\eta'_{j}\right)-\phi_{j}\left(\varepsilon'_{j}\right)=g_{j}(\alpha_{j}).$
Hence, if (\ref{eq:Goal-P}) is to hold for all simple intervals at
stage $k,$ we need (\ref{eq:Goal-Alpha}) to hold for all sequences
$\alpha_{j},$ such that $\sum_{j=1}^{k}\alpha_{j}m_{j+1}m_{j+2}\cdots m_{k}\geq0.$
With this notation we can restate Assumptions 1 and 2a as: 
\begin{description}
\item [{Assumption-1'}] $\left(\left(1+\alpha_{k'}\right)^{t}-1\right)\left(m_{k'+1}\cdots m_{k}\right)^{t}\leq g_{k'}\left(\alpha_{k'}\right)\frac{b(k')}{b(k)},$
if $0<\alpha_{k'}$ and $1\leq k'\leq k$.
\item [{Assumption-2'}] $\left(\alpha_{k'}m_{k'+1}\cdots m_{k}\right){}^{t}\leq\left(g_{k'}\left(\alpha_{k'}\right)-1\right)\frac{b(k')}{b(k)},$
if $0<\alpha_{k'}$ and $1\leq k'<k$.
\end{description}
The following elementary inequality is useful.
\begin{lem}
\label{lem:Magic-Inequality}The inequality 
\[
\left(A+BC\right)^{t}\leq A^{t}+\left(\left(1+B\right)^{t}-1\right)C^{t},
\]
holds for $1\leq t,$ $0\leq A\leq C$ and $0\leq B.$ \end{lem}
\begin{proof}
Let $\varphi(A):=\left(A+BC\right)^{t}$ and $\psi(A):=A^{t}+\left(\left(1+B\right)^{t}-1\right)C^{t}.$
Then $\varphi(C)=\psi(C),$ $\varphi'(A)=t\left(A+BC\right)^{t-1},$
and $\psi'(A)=tA^{t-1}.$ Since $BC\geq0$ and $t\geq1,$ it follows
that $\varphi'(A)\geq\psi'(A).$ Consequently, $\varphi(A)\leq\psi(A)$
for all $A\leq C.$ Thus the stated inequality holds.
\end{proof}

\subsection{Proof of Lemma \ref{lem:Basic-vs-Simple-Intervals} under Assumptions
1 and 2b}

For the purposes of our proof, it is convenient to consider positive
values of $\alpha_{j}$. The following Lemma shows that this assumption
is sufficient to prove the result.
\begin{lem}
\label{lem:Positive-alphas-only} Let $D_{k}$ be a sequence of sets
such that $d_{k,j+1}-d_{k,j}\ge2$ for all $k\in\mathbb{N}$ and $0\le j<m_{k}-1$.
For any simple interval $P$ at stage $k$ there exists a simple interval
$Q$ at stage $k$ such that $\left|Q\right|\le\left|P\right|$, $i_{P}=i_{Q}$,
and $\alpha'_{j}\ge0$ for all $1\le j\le k$.\end{lem}
\begin{proof}
Let $P=[x,y],$ where $x<y$ are as in (\ref{eq:Left-End-Point})
and (\ref{eq:Right-End-Point}). The result is trivial, if $\alpha_{j}\ge0$
for all $1\le j\le k$. Let $k'$ be the smallest subscript for which
$\alpha_{k'}=\eta_{k'}-\varepsilon_{k'}<0$, then $\alpha_{j}\ge0$
for all $1\le j<k'\le k$. We will show that $Q=\left[x',y'\right]$
can be constructed by induction on $k'$.

Let $k''$ denote the largest subscript less than $k'$ such that
$\alpha_{k''}>0$. Since $x<y$ such a subscript exits. Then $\alpha_{j}=0$
for $j=k''+1,\ldots,k'-1.$ 

Let $x':=\sum_{j=1}^{k}\phi_{j}\left(\varepsilon'_{j}\right)b(j)$
and $y':=b(k)+\sum_{j=1}^{k}\phi_{j}\left(\eta'_{j}\right)b(j),$
where 
\begin{align*}
\varepsilon'_{j}:=\begin{cases}
\varepsilon_{j} & \text{ if }1\le j\le k''\\
0 & \text{ if }k''<j\le k'\\
\varepsilon_{j} & \text{ if }k'<j\le k
\end{cases} & \quad\text{and} & \eta'_{j}:=\begin{cases}
\eta_{j} & \text{ if }1\le j<k''\\
\eta_{j}-1 & \text{ if }j=k''\\
m_{j}-1 & \text{ if }k''<j<k'\\
m_{j}+\alpha_{j} & \text{ if }j=k'\\
\eta_{j} & \text{ if }k'<j\le k
\end{cases}
\end{align*}
and let $Q:=\left[x',y'\right].$ By construction $Q$ is a simple
interval at stage $k.$ Clearly, $\alpha'_{j}\ge0$ for all $1\le j\le k'\le k$.
Since, 
\begin{align*}
 & \alpha_{k''}m_{k''+1}\cdots m_{k}+\alpha_{k'}m_{k'+1}\cdots m_{k}\\
 & =\left(\alpha_{k''}-1\right)m_{k''+1}\cdots m_{k}+\left(\sum_{j=k''+1}^{k'-1}\left(m_{j}-1\right)m_{j+1}\cdots m_{k}\right)\\
 & \quad+\left(m_{k'}+\alpha_{k'}\right)m_{k'+1}\cdots m_{k}.
\end{align*}
we conclude $i_{Q}=i_{P}.$ Finally, since $g_{k''}\left(1\right)\ge2,$
$\phi_{j}\left(m_{j}-1\right)\leq\frac{1-\beta_{j}}{\beta_{j}}$,
$-\frac{1-\beta_{k'}}{\beta_{k'}}\leq\phi_{k'}\left(\eta_{k'}\right)-\phi_{k'}\left(\varepsilon_{k'}\right)<0,$
and $0\leq\phi_{k'}\left(m_{k'}+\alpha_{k'}\right)\leq\frac{1-\beta_{k'}}{\beta_{k'}},$
then
\begin{align*}
|P|-|Q| & =\left(\phi_{k''}\left(\eta_{k''}\right)-\left(\phi_{k''}\left(\eta{}_{k''}-1\right)\right)\right)b\left(k''\right)-\left(\sum_{j=k''+1}^{k'-1}\phi_{j}\left(m_{j}-1\right)b\left(j\right)\right)\\
 & \quad+\left(\phi_{k'}\left(\eta_{k'}\right)-\phi_{k'}\left(\varepsilon_{k'}\right)-\phi_{k'}\left(m_{k'}+\alpha_{k'}\right)\right)b\left(k'\right)\\
 & \ge2b\left(k''\right)-\left(\sum_{j=k''+1}^{k'-1}\frac{1-\beta_{j}}{\beta_{j}}b\left(j\right)\right)-2\frac{1-\beta_{k'}}{\beta_{k'}}b\left(k'\right)\\
 & =2b\left(k''\right)-\left(b\left(k''\right)-b\left(k'-1\right)\right)-2\left(1-\beta_{k'}\right)b\left(k'-1\right)\\
 & =b\left(k''\right)-b\left(k'-1\right)+2b\left(k'\right)\\
 & \geq2b\left(k'\right).
\end{align*}
 Hence, $\left|Q\right|\le\left|P\right|$. 
\end{proof}
We now have the tools to prove Lemma \ref{lem:Basic-vs-Simple-Intervals}:
\begin{proof}[Proof of Lemma \ref{lem:Basic-vs-Simple-Intervals}]
According to Lemma \ref{lem:Positive-alphas-only}, it is sufficient
to consider simple intervals $P$ such that $\alpha_{j}\ge0$ for
all $1\le j\le k$. If $\alpha_{j}=0$ for all $1\leq j\leq k,$ then
(\ref{eq:Goal-Alpha}) simplifies to $1^{t}\leq1.$ Hence we will
assume $\alpha_{j}\neq0$ for some $j.$ Let $k'$ be the smallest
subscript for which $\mbox{\ensuremath{\alpha}}_{k'}\neq0$. 

For each $k'\le i<k$, let

\begin{align*}
A_{i} & :=1+\sum_{j=i+1}^{k}\alpha_{j}m_{j+1}\cdots m_{k},\\
B_{i} & :=\alpha_{i},\text{ and }\\
C_{i} & :=m_{i+1}\cdots m_{k}.
\end{align*}

Since $A_{i}\leq C_{i}$ and $B_{i}=\alpha_{i}\ge0$ for all $k'\le i<k$,
we can apply Lemma \ref{lem:Magic-Inequality} so that

\begin{align*}
\left(1+\sum_{j=k'}^{k}\alpha_{j}m_{j+1}\cdots m_{k}\right)^{t} & =\left(1+\alpha_{k'}m_{k'+1}\cdots m_{k}+\sum_{j=k'+1}^{k}\alpha_{j}m_{j+1}\cdots m_{k}\right)^{t}\\
 & \leq\left(1+\sum_{j=k'+1}^{k}\alpha_{j}m_{j+1}\cdots m_{k}\right)^{t}\\
 & \quad+\left(\left(1+\alpha_{k'}\right)^{t}-1\right)\left(m_{k'+1}\cdots m_{k}\right)^{t}.
\end{align*}
Continuing in this manner, after a finite number of substitutions
we obtain:
\begin{align*}
\left(1+\sum_{j=i}^{k}\alpha_{j}m_{j+1}\cdots m_{k}\right)^{t} & \leq\left(1+\alpha_{k}\right)^{t}+\sum_{j=k'}^{k-1}\left(\left(1+\alpha_{j}\right)^{t}-1\right)\left(m_{j+1}\cdots m_{k}\right)^{t}\\
 & \le1+g_{k}\left(\alpha_{k}\right)+\sum_{j=k'}^{k-1}g_{j}\left(\alpha_{j}\right)\frac{\beta(j)}{\beta(k)},
\end{align*}
where the last step requires Assumption 1. Hence, equation (\ref{eq:Goal-Alpha})
is satisfied.
\end{proof}

\subsection{Proof of Lemma \ref{lem:Basic-vs-Simple-Intervals} under Assumptions
1 and 2a}

If $\alpha_{j}=0$ for all $1\leq j\leq k,$ then (\ref{eq:Goal-Alpha})
simplifies to $1^{t}\leq1.$ Hence, we will assume $\alpha_{j}\neq0$
for some $j.$ Let $k'$ be the smallest subscript for which $\mbox{\ensuremath{\alpha}}_{k'}\neq0.$
Then $x<y$ implies $\alpha_{k'}>0.$ We will show by induction on
$k',$ that if $\alpha_{k'}>0,$ then 
\begin{align}
 & \left(1+\alpha_{k'}m_{k'+1}\cdots m_{k}+\sum_{j=k'+1}^{k}\alpha_{j}m_{j+1}\cdots m_{k}\right)^{t}\label{eq:Induction-Goal}\\
 & \leq1+g_{k'}\left(\alpha_{k'}\right)\frac{b(k')}{b(k)}+\sum_{j=k'+1}^{k}g_{j}\left(\alpha_{j}\right)\frac{b(j)}{b(k)}.\nonumber 
\end{align}
The basis case is $k'=k.$ In this case we must show 
\begin{equation}
\left(1+\alpha_{k}\right)^{t}\leq1+g_{k}\left(\alpha_{k}\right).\label{eq:Assumption-0}
\end{equation}
But, this is Assumption 1', with $k'=k.$ 

Inductively, suppose 
\begin{align}
 & \left(1+\alpha_{\widetilde{k}}m_{\widetilde{k}+1}\cdots m_{k}+\sum_{j=\widetilde{k}+1}^{k}\alpha_{j}m_{j+1}\cdots m_{k}\right)^{t}\label{eq:Ind-Hyp}\\
 & \leq1+g_{\widetilde{k}}\left(\alpha_{\widetilde{k}}\right)\frac{b\left(\widetilde{k}\right)}{b(k)}n^{k-\widetilde{k}}+\sum_{j=\widetilde{k}+1}^{k}g_{j}\left(\alpha_{j}\right)\frac{b(j)}{b(k)}.\nonumber 
\end{align}
holds for all $\widetilde{k}>k'$ with $\alpha_{\widetilde{k}}>0$$.$
We must show that (\ref{eq:Induction-Goal}) holds assuming $\alpha_{k'}>0.$ 

Either $\alpha_{j}=0$ for all $k'<j\leq k,$ or there is a smallest
$k'',$ such that $k'<k''\leq k$ and $\alpha_{k''}\neq0.$ Hence,
we will consider the three cases: 
\begin{enumerate}
\item $\alpha_{j}=0$ for all $k'<j\leq k,$ 
\item $\alpha_{j}=0$ for all $k'<j<k''\leq k,$ and $\alpha_{k''}>0,$
and 
\item $\alpha_{j}=0$ for all $k'<j<k''\leq k,$ and $\alpha_{k''}<0.$ 
\end{enumerate}
Note, $k''=1+k'$ is possible. Each case requires a separate argument. 

\textbf{Case 1.} In this case (\ref{eq:Goal-Alpha}) simplifies to
\begin{equation}
\left(1+\alpha_{k'}m_{k'+1}m_{k'+2}\cdots m_{k}\right)^{t}\leq1+g_{k'}\left(\alpha_{k'}\right)\frac{b(k')}{b(k)}.\label{eq:Assumption-1}
\end{equation}
 Setting $A=1,$ $B=\alpha_{k},$ and $C=m_{k'+1}\cdots m_{k}$ in
Lemma \ref{lem:Magic-Inequality} and using Assumption 1' we see that
\begin{align*}
 & \left(1+\alpha_{k'}m_{k'+1}m_{k'+2}\cdots m_{k}\right)^{t}\\
 & \leq1+\left(\left(1+\alpha_{k'}\right)^{t}-1\right)\left(m_{k'+1}m_{k'+2}\cdots m_{k}\right)^{t}\\
 & \leq1+g_{k'}\left(\alpha_{k'}\right)\frac{b\left(k'\right)}{b(k)}.
\end{align*}
This established (\ref{eq:Assumption-1}).

\textbf{Case 2.} In this case (\ref{eq:Goal-Alpha}) can be restated
as 
\begin{align}
 & \left(1+\alpha_{k'}m_{k'+1}\cdots m_{k}+\alpha_{k''}m_{k''+1}\cdots m_{k}+\sum_{j=k''+1}^{k}\alpha_{j}m_{j+1}\cdots m_{k}\right)^{t}\label{Sec-5-eq:Case-2-goal}\\
 & \leq1+g_{k'}\left(\alpha_{k'}\right)\frac{b\left(k'\right)}{b(k)}+g_{k''}\left(\alpha_{k''}\right)\frac{b\left(k''\right)}{b(k)}+\sum_{j=k''+1}^{k}g_{j}\left(\alpha_{j}\right)\frac{b\left(j\right)}{b(k)}\nonumber 
\end{align}
where $\alpha_{k''}>0$ and $g_{k''}\left(\alpha_{k''}\right)>0.$ 

Using Lemma \ref{lem:Magic-Inequality} with
\begin{align*}
A & =1+\alpha_{k''}m_{k''+1}\cdots m_{k}+\sum_{j=k''+1}^{k}\alpha_{j}m_{j+1}\cdots m_{k},\\
B & =\alpha_{k'}\\
C & =m_{k'+1}\cdots m_{k}
\end{align*}
we have the first inequality in 
\begin{align*}
 & \left(1+\alpha_{k'}m_{k'+1}\cdots m_{k}+\alpha_{k''}m_{k''+1}\cdots m_{k}+\sum_{j=k''+1}^{k}\alpha_{j}m_{j+1}\cdots m_{k}\right)^{t}\\
 & \leq\left(1+\alpha_{k''}m_{k''+1}\cdots m_{k}+\sum_{j=k''+1}^{k}\alpha_{j}m_{j+1}\cdots m_{k}\right)^{t}\\
 & \quad+\left(\left(1+\alpha_{k'}\right)^{t}-1\right)\left(m_{k'+1}\cdots m_{k}\right)^{t}\\
 & \leq1+g_{k''}\left(\alpha_{k''}\right)\frac{b\left(k''\right)}{b(k)}+\sum_{j=k''+1}^{k}g_{j}\left(\alpha_{j}\right)\frac{b\left(j\right)}{b(k)}+\left(\left(1+\alpha_{k'}\right)^{t}-1\right)\left(m_{k'+1}\cdots m_{k}\right)^{t}.
\end{align*}
The second inequality follows from the inductive hypothesis (\ref{eq:Ind-Hyp})
with $\widetilde{k}=k''.$ Hence, (\ref{Sec-5-eq:Case-2-goal}) follows
from Assumption 1'.

\textbf{Case 3.} As in Case 2 (\ref{eq:Goal-Alpha}) can be restated
as (\ref{Sec-5-eq:Case-2-goal}), but now, $\alpha_{k''}<0$ and $g_{k''}\left(\alpha_{k''}\right)<0.$
Then 
\begin{align*}
 & 1+\alpha_{k''}m_{k''+1}\cdots m_{k}+\sum_{j=k''+1}^{k}\alpha_{j}m_{j+1}\cdots m_{k}\\
 & \leq1+\alpha_{k''}m_{k''+1}\cdots m_{k}+\sum_{j=k''+1}^{k}\left(m_{j}-1\right)m_{j+1}\cdots m_{k}\\
 & =\left(\alpha_{k''}+1\right)m_{k''+1}\\
 & \leq0.
\end{align*}
So 
\begin{align*}
0 & \leq\left(1+\alpha_{k'}m_{k'+1}\cdots m_{k}+\alpha_{k''}m_{k''+1}\cdots m_{k}+\sum_{j=k''+1}^{k}\alpha_{j}m_{j+1}\cdots m_{k}\right)^{t}\\
 & \leq\left(\alpha_{k'}m_{k'+1}\cdots m_{k}\right)^{t}.
\end{align*}
On the other hand, since $g_{j}\left(\alpha_{j}\right)\geq-\frac{1-\beta_{j}}{\beta_{j}}$
for $k''\leq j\leq k,$ we have 
\begin{align*}
 & 1+g_{k'}\left(\alpha_{k'}\right)\frac{b\left(k'\right)}{b(k)}+g_{k''}\left(\alpha_{k''}\right)\frac{b\left(k''\right)}{b(k)}+\sum_{j=k''+1}^{k}g_{j}\left(\alpha_{j}\right)\frac{b\left(j\right)}{b(k)}\\
 & \geq1+g_{k'}\left(\alpha_{k'}\right)\frac{b\left(k'\right)}{b(k)}-\sum_{j=k''}^{k}\frac{1-\beta_{j}}{\beta_{j}}\frac{b\left(j\right)}{b(k)}\\
 & =g_{k'}\left(\alpha_{k'}\right)\frac{b\left(k'\right)}{b(k)}-\frac{b\left(k''-1\right)}{b(k)}\\
 & \geq\left(g_{k'}\left(\alpha_{k'}\right)-1\right)\frac{b\left(k'\right)}{b(k)}.
\end{align*}
Hence, (\ref{Sec-5-eq:Case-2-goal}) follows from Assumption 2'. This
completes the proof of Lemma \ref{lem:Basic-vs-Simple-Intervals}.

\section{Examples\label{sec-6:Examples}}

In this section we construct some examples illustrating various aspects
of our results. To facilitate the construction of the examples we
restate Assumption-1 and Assumption-2a in simpler forms, see Lemma
\ref{lem:A-1-Alternative} and Lemma \ref{lem:A-2-Alternative}. These
lemmas are also used in the proof of Theorem \ref{thm:Main}. We also
investigate some of the consequences of Assumption-1, -2a, and -2b.

\subsection{Assumption-1}

In this section we investigate Assumption-1. We begin by re-stating
Assumption-1 in a simpler form. The first consequence of this is that
the limit inferior used in (\ref{eq:Haudorff-dim-formula}) is in
fact a limit. Another consequence is that $\beta_{k}^{s}m_{k}$ is
``nearly'' equal to one, see Lemma \ref{Sec-6-lem:6.3}, Corollary
\ref{cor: Sets of density zero}, and Corollary \ref{cor:Constant-beta}
for interpretations of the term ``nearly''. We show, Corollary \ref{cor:Constant-beta},
that, if $\beta_{k}=\beta,$ then the extreme points $0$ and $\frac{1-\beta}{\beta}$
must be in $D_{k}$ for ``nearly'' all $k.$ We conclude this section
with an example illustrating that the dimension $s$ can be close
to one, when Assumption-1 holds. 
\begin{lem}
\label{lem:A-1-Alternative}Assumption-1 is equivalent to 
\begin{equation}
\beta_{k}^{s}m_{k}\leq1,\label{eq:0-beta}
\end{equation}
 and 
\begin{equation}
\left(1+j-i\right)^{t}-1\leq d_{k,j}-d_{k,i}\label{eq:1-beta}
\end{equation}
 for all $k\geq1$ and all $0\leq i<j<m_{k}.$ \end{lem}
\begin{proof}
Suppose Assumption-1 holds. Setting $k'=k$ in Assumption-1 shows
that (\ref{eq:1-beta}) holds. Setting $k'=k,$ $i=0$ and $j=m_{k}-1$
in Assumption-1 gives 
\[
m_{k}^{t}-1\leq d_{k,m_{k}-1}-d_{k,0}.
\]
Since $d_{k,0}$ and $d_{k,m_{k}-1}$ are in $\left[0,\frac{1-\beta_{k}}{\beta_{k}}\right]$
also $d_{k,m_{k}-1}-d_{k,0}\leq\frac{1-\beta_{k}}{\beta_{k}}.$ Hence,
$m_{k}^{t}-1\leq\frac{1-\beta_{k}}{\beta_{k}}$ and therefore 
\[
m_{k}^{t}\leq\frac{1}{\beta_{k}}.
\]
This established (\ref{eq:0-beta}). Conversely, it is easy to see
that (\ref{eq:0-beta}) and (\ref{eq:1-beta}) implies Assumption-1. \end{proof}
\begin{cor}
If Assumption-1 holds, then the limit inferior in (\ref{eq:Haudorff-dim-formula})
and in (\ref{Sec-2-eq:Def-of-L}) are limits.\end{cor}
\begin{proof}
It follows from (\ref{eq:0-beta}) that 
\[
\frac{\log\left(m_{1}\cdots m_{j}\right)}{-\log\left(\beta_{1}\cdots\beta_{j}\right)}\leq\frac{\log\left(m_{1}\cdots m_{j}\right)}{\log\left(m_{1}^{t}\cdots m_{j}^{t}\right)}=\frac{1}{t}=s.
\]
Hence, the limit inferior in (\ref{eq:Haudorff-dim-formula}) is a
limit. 

By (\ref{eq:Haudorff-dim-formula}) the $\mu(k)b(k)^{s}$ is decreasing.
Hence, the limit inferior in (\ref{Sec-2-eq:Def-of-L}) is a limit. \end{proof}
\begin{rem}
Clearly, (\ref{eq:0-beta}) implies $L\leq1.$ Since $\prod_{k=1}^{\infty}\left(1+p_{k}\right)$
converges to a positive value if and only if $\sum_{k=1}^{\infty}p_{k}$
is absolutely convergent, it follows from (\ref{eq:0-beta}) that
$0<L$ if and only if $\sum_{k=1}^{\infty}\left(1-m_{k}\beta_{k}^{s}\right)$
is convergent. 
\end{rem}
Our next few results, Lemma \ref{Sec-6-lem:6.3}, Corollary \ref{cor: Sets of density zero},
and Corollary \ref{cor:Constant-beta} show that, when Assumption-1
holds, there is ``nearly'' equality in (\ref{eq:0-beta}). 
\begin{lem}
\label{Sec-6-lem:6.3}For $0<\xi<1$ let $F_{\xi,k}$ be the integers
$1\leq j\leq k$ for which $\beta_{j}^{s}m_{j}<\xi$. If Assumption-1
holds, then
\begin{equation}
\lim_{k\to\infty}\frac{\#F_{\xi,k}}{\sum_{j=1}^{k}\log\left(\frac{1}{\beta_{j}}\right)}=0.\label{eq:2-beta}
\end{equation}
\end{lem}
\begin{proof}
Let $E_{\xi,k}$ be the integers $1\leq j\leq k$ for which $\xi\leq\beta_{j}^{s}m_{j}.$
Then $E_{\xi,k}$ and $F_{\xi,k}$ are disjoint and their union is
the set $\left\{ 1,2,\ldots,k\right\} .$ It follows that 
\begin{align*}
s & =\liminf_{k\to\infty}\frac{\log\left(m_{1}\cdots m_{k}\right)}{\log\left(\frac{1}{\beta_{k}}\cdots\frac{1}{\beta_{k}}\right)}\\
 & =\liminf_{k\to\infty}\frac{\sum_{j\in E_{\xi,k}}\log\left(m_{j}\right)+\sum_{j\in F_{\xi,k}}\log\left(m_{j}\right)}{\sum_{j=1}^{k}\log\left(\frac{1}{\beta_{j}}\right)}\\
 & \leq\liminf_{k\to\infty}\frac{\sum_{j\in E_{\xi,k}}\log\left(\frac{1}{\beta_{j}^{s}}\right)+\sum_{j\in F_{\xi,k}}\log\left(\frac{\xi}{\beta_{j}^{s}}\right)}{\sum_{j=1}^{k}\log\left(\frac{1}{\beta_{j}}\right)}\\
 & =s+\log\left(\xi\right)\limsup_{k\to\infty}\frac{\#F_{\xi,k}}{\sum_{j=1}^{k}\log\left(\frac{1}{\beta_{j}}\right)}.
\end{align*}
Where we used $m_{j}\leq\frac{1}{\beta_{j}^{s}},$ for $j$ in $E_{\xi,k}$
by (\ref{eq:0-beta}) and $m_{j}\leq\frac{\xi}{\beta_{j}^{s}}$ for
$\xi$ in $F_{\xi,k}$ by the definition $F_{\xi,k}.$ Using $\log(\xi)<0$,
(\ref{eq:2-beta}) follows from this calculation. 
\end{proof}
A subset $A$ of $\mathbb{N}$ has \emph{density} $\gamma,$ if 
\[
\frac{\#\left(A\cap[1,k]\right)}{k}\to\gamma\text{ as }k\to\infty.
\]
Here $\#B$ denotes the cardinality of the finite set $B.$ We say
a sequence $\left(a_{k}\right)$ is \emph{nearly constantly equal
to} $a$, if the set $A:=\left\{ k\mid a_{k}\neq a\right\} $ has
density zero for some $a.$ 
\begin{cor}
\label{cor: Sets of density zero}If Assumption-1 holds and there
is a $\beta>0,$ such that $\beta\leq\beta_{j}$ for all $j,$ then
the set 
\[
A_{\xi}:=\left\{ k\mid\beta_{k}^{s}m_{k}\leq\xi\right\} 
\]
has density zero for any $0<\xi<1$. Furthermore, if $\beta_{j}=\beta$
for all $j,$ then 
\[
A_{1}':=\left\{ k\mid\beta^{s}m_{k}<1\right\} 
\]
has density zero. Hence, $\left(m_{k}\right)$ is nearly constantly
equal to $M:=\frac{1}{\beta^{s}}$ and consequently, 
\begin{equation}
s=\frac{\log\left(M\right)}{-\log\left(\beta\right)}.\label{eq:dim-const-beta}
\end{equation}
\end{cor}
\begin{proof}
Since $\frac{\#F_{\xi,k}}{\sum_{j=1}^{k}\log\left(\frac{1}{\beta_{j}}\right)}\geq\frac{\#F_{\xi,k}}{\sum_{j=1}^{k}\log\left(\frac{1}{\beta}\right)}$
it follows from (\ref{eq:2-beta}) that $A_{\xi}$ has density zero.
If $\beta_{j}=\beta$ for all $j,$ then there is a $\xi$ such that
$A_{1}'=A_{\xi}.$ 

Let $E$ be the set of integers $j$ for which $m_{j}=M=\frac{1}{\beta^{s}}.$
Then $E$ is the complement of $A_{1}'$ and $m_{j}<M$ for $j\in A_{1}'.$
\begin{align*}
0\leq & \frac{\sum_{j=1}^{k}\log\left(m_{j}\right)}{k\log\left(\frac{1}{\beta}\right)}-\frac{\sum_{j\in E\cap[1,k]}\log\left(m_{j}\right)}{k\log\left(\frac{1}{\beta}\right)}\\
 & \leq\frac{\sum_{j\in A_{1}'\cap[1,k]}\log\left(M\right)}{k\log\left(\frac{1}{\beta}\right)}\\
 & \to0,
\end{align*}
since $A_{1}'$ has density zero. Since $\frac{\sum_{j=1}^{k}\log\left(m_{j}\right)}{k\log\left(\frac{1}{\beta}\right)}\to s$
and 
\[
\frac{\sum_{j\in E\cap[1,k]}\log\left(m_{j}\right)}{k\log\left(\frac{1}{\beta}\right)}=\frac{k\log\left(M\right)-\#\left(A_{1}'\cap[1,k]\right)}{k\log\left(\frac{1}{\beta}\right)}\to\frac{\log\left(M\right)}{\log\left(\frac{1}{\beta}\right)}
\]
we have established (\ref{eq:dim-const-beta}). 
\end{proof}
If $\beta_{k}=\beta,$ then the $D_{k}$ must contain the extreme
points $0$ and $\frac{1-\beta}{\beta}$ for nearly all $k.$ 
\begin{cor}
\label{cor:Constant-beta}If Assumption-1 holds and $\beta_{k}=\beta$
for all $k,$ then $d_{k,0}=0$ and $d_{k,M-1}=\frac{1-\beta}{\beta}$
for nearly all $k,$ where $M$ is as in Corollary \ref{cor: Sets of density zero}. \end{cor}
\begin{proof}
By Corollary \ref{cor: Sets of density zero} $m_{k}=M=\frac{1}{\beta^{s}}$
and $\beta^{s}M=1$ for nearly all $k.$ So by (\ref{eq:1-beta})$ $
$M^{t}\leq1+d_{k,0}-d_{k,M-1}.$ Since, $M^{t}-1=\frac{1-\beta}{\beta}$
and $D_{k}$ is a subset of the interval $\left[0,\frac{1-\beta}{\beta}\right]$
the claim follows. \end{proof}
\begin{rem}
The proof of Corollary \label{Sec-6-Rem:Constant-beta-and-m} shows
that, if $\beta_{k}=\beta$ and $m_{k}=M$ for all $k,$ then $d_{k,0}=0$
and $d_{k,M-1}=\frac{1-\beta}{\beta}$ for all $k.$ 
\end{rem}
If $m_{k}=2$ and $D_{k}$ consists of the extreme points then Assumption-1
is equivalent to (\ref{eq:0-beta}).
\begin{cor}
\label{cor:Middle 1-2gamma sets}If $m_{k}=2$ and $D_{k}=\left\{ 0,\frac{1-\beta_{k}}{\beta_{k}}\right\} $
for all $k\in\mathbb{N}$, then Assumption-1 is equivalent to $\beta_{k}^{s}\le\frac{1}{2}$
for all $k$.\end{cor}
\begin{proof}
It is sufficient to show that equations (\ref{eq:1-beta}) and (\ref{eq:0-beta})
are equivalent. Since $m_{j}=2$ for all $j$, then $\eta_{j},\varepsilon_{j}\in\left\{ 0,1\right\} $
and $\varepsilon_{j}<\eta_{j}$ iff $\eta_{j}-\varepsilon_{j}=1$.
Thus, equation (\ref{eq:1-beta}) simplifies to $2^{t}\le1+d_{j,1}-d_{j,0}=\frac{1}{\beta_{k}}$,
but this is (\ref{eq:0-beta}).
\end{proof}
The following is an example of a Cantor set $C$ satisfying Assumption-1
and with $s=1.$ 
\begin{example}
Let $\beta_{j}=\frac{1}{2}\gamma_{j}$ where $0<\gamma_{j}<1.$ Set
$D_{j}:=\left\{ 0,d_{j}\right\} ,$ where $1\leq d_{j}\leq\frac{1-\beta_{j}}{\beta_{j}},$
for all $j.$ In particular, $m_{j}=2$ for all $j.$ With this notation
\begin{align*}
\frac{\log\left(m_{1}\cdots m_{k}\right)}{\log\left(\frac{1}{\beta_{1}}\cdots\frac{1}{\beta_{k}}\right)} & =\frac{\log\left(2\right)}{\log\left(2\right)+\log\left(\gamma_{1}\cdots\gamma_{k}\right)^{1/k}}.
\end{align*}
Hence, if $\left(\gamma_{1}\cdots\gamma_{k}\right)^{1/k}\to1,$ e.g.,
if $\gamma_{k}=\frac{k}{k+1},$ then $s=1.$ Note $t=\frac{1}{s}=1.$
Since $\beta_{j}<\frac{1}{2}$ we have verified (\ref{eq:0-beta}).
Since $d_{j}-0\geq1$ the inequality (\ref{eq:1-beta}) holds. It
follows from Lemma \ref{lem:A-1-Alternative} that Assumption-1 is
satisfied. 
\end{example}

\subsection{Assumption-1 and Assumption-2}

In this section we investigate the relationship between our separation
assumptions: Assumption-2a and Assump\-tion-2b. 
\begin{lem}
\label{lem:A-2-Alternative}If (\ref{eq:0-beta}) holds, in particular,
if Assumption-1 holds, then Assumption-2a is equivalent to 
\begin{equation}
1+\left(j-i\right)^{t}\beta_{k+1}m_{k+1}^{t}\leq d_{k,j}-d_{k,i}\label{eq:A-2-alternative}
\end{equation}
for all $k$ and all $0\leq i<j<m_{k}.$ \end{lem}
\begin{proof}
Setting $j=k'$ and $k=k'+1$ in Assumption-2a gives (\ref{eq:A-2-alternative}).
Conversely, combining (\ref{eq:0-beta}) and (\ref{eq:A-2-alternative})
leads to Assumption-2a.\end{proof}
\begin{rem}
Setting $j-i=1$ in (\ref{eq:A-2-alternative}) gives 
\begin{equation}
m_{k+1}^{t}\beta_{k+1}\leq g_{k}\left(1\right)-1.\label{eq:A-2-alpha=00003D1}
\end{equation}

(\emph{a}) If $\beta_{j}=\beta$ for all $j,$ then the set of $k$
such that $m_{k+1}^{t}\beta_{k+1}<1$ has density zero, by Corollary
\ref{cor: Sets of density zero}. Hence, using (\ref{eq:A-2-alpha=00003D1})
we see that, except for a set of $k$ with density zero, we have $2\leq d_{k,j+1}-d_{k,j}$
for all $j=0,\ldots,m_{k}-2.$ 

(\emph{b}) If $0<\beta\leq\beta_{j}$ for all $j,$ then for any $0<\xi<1,$
the set of $k$ such that $m_{k+1}^{t}\beta_{k+1}<\xi$ has density
zero, by Corollary \ref{cor: Sets of density zero}. Hence, using
(\ref{eq:A-2-alpha=00003D1}) we see that, except for a set of $k$
with density zero, we have $1+\xi\leq d_{k,j+1}-d_{k,j}$ for all
$j=0,\ldots,m_{k}-2.$ 
\end{rem}
The following results shows that if $s\leq\log(2)/\log(3),$ then
Assumption-1 implies Assumption-2a and Assumption-2b. Hence, to construct
an example where Assumption-1 holds and one or both of Assumption-2a
and Assumption-2b fails, we must consider $s>\log(2)/\log(3).$ 
\begin{cor}
\label{cor:A1 with small s}If Assumption-1 holds and $s\leq\frac{\log\left(2\right)}{\log\left(3\right)},$
then Assumption-2a and Assumption-2b both hold. \end{cor}
\begin{proof}
Fix $0\leq j\leq m_{k}-2.$ Setting $\varepsilon_{k}=j$ and $\eta_{k}=1+j$
in (\ref{eq:1-beta}) gives 
\[
2^{t}\leq1+d_{k,j+1}-d_{k,j}.
\]
Since $s\leq\frac{\log\left(2\right)}{\log\left(3\right)}$ is equivalent
to $3\leq2^{t},$ it follows that $2\leq d_{k,j+1}-d_{k,j}.$ It follows
that Assumption-2b holds. 

Let $\alpha=j-i.$ Using $3\leq2^{t}$ and $1\leq\alpha$ it is easy
to see that 
\[
1+\alpha^{t}\leq\left(1+\alpha\right)^{t}-1.
\]
Using (\ref{eq:1-beta}) we see that 
\[
1+\left(j-i\right)^{t}\leq\left(1+\alpha\right)^{t}-1\leq d_{k,j}-d_{k,i}.
\]
If follows from (\ref{eq:0-beta}) that (\ref{eq:A-2-alternative})
holds. So Lemma \ref{lem:A-2-Alternative} shows that Assumption-2a
holds. 
\end{proof}
In general, not all sets satisfying Assumption-1 and Assumption-2a
have dimension $s\le\frac{\log\left(2\right)}{\log\left(3\right)}$
as required by Corollary \ref{cor:A1 with small s}. Examples \ref{Ex: A1+A2 not necessarily separate}
and \ref{Sec-6-Example:A1-A2a-not-A2b} constructs a families of sets
satisfying both Assumption-1 and Assumption-2a, but not Assumption-2b.
In one example we set $m_{k}=3$ for all $k$ and vary $\beta_{k}.$
In the other example we fix $\beta_{k}=\beta$ and vary $m_{k}.$ 
\begin{example}
\label{Ex: A1+A2 not necessarily separate}In this example we use
some of the notation introduced early in Section \ref{sec:Proof-of-Lemma-Basic-vs-Simple}.
Choose $0<\beta_{k}\le\frac{1}{5}$ such that $A:=\left\{ k\mid\beta_{k}<\frac{1}{5}\right\} $
is a set of density zero and suppose 
\[
D_{k}:=\left\{ 0,1+\xi_{k},\frac{1-\beta_{k}}{\beta_{k}}\right\} 
\]
Thus, $m_{k}=3$ for all $k$ so that $\alpha_{k}=1,2$ and $s=\frac{\log\left(3\right)}{\log\left(5\right)}>\frac{\log\left(2\right)}{\log\left(3\right)}$.
Since $g_{k}\left(1\right)$ represents a minimum difference, we will
assume $\left(1+\xi_{k}\right)-0\le\frac{1-\beta_{k}}{\beta_{k}}-\left(1+\xi_{k}\right)$
for all $k$ without loss of generality. We will give constraints
on $\xi_{k}$ such that $D_{k}$ satisfies both Assumption-1 and Assumption-2a. 

Equation (\ref{eq:0-beta}) is satisfied whenever $\beta_{k}\le\frac{1}{5}$.
Since $g_{k}\left(2\right)=\frac{1-\beta_{k}}{\beta_{k}}\ge4$ for
any value of $\beta_{k}$, then $D_{k}$ satisfies Assumption-1 whenever
$g_{k}\left(1\right)=1+\xi_{k}\ge2^{t}-1$.

Let $\left(m_{k}\right)^{t}\beta_{k}=\gamma_{k}\le1$ so that for
any $k>k'$ Assumption-2a can be written 
\[
g_{k'}\left(\alpha_{k'}\right)-1\ge\left(\alpha_{k'}\right)^{t}\gamma_{k'+1}\ge\left(\alpha_{k'}\right)^{t}\gamma_{k'+1}\cdots\gamma_{k}
\]
According to Lemma \ref{lem:A-2-Alternative}, it is sufficient to
consider only $k=k'+1$. Since $g_{k-1}\left(2\right)-1\ge3>2^{t}\ge2^{t}\gamma_{k}$,
then $D_{k-1}$ satisfies Assumption-2a whenever
\[
\xi_{k-1}=g_{k-1}\left(1\right)-1\ge\left(1\right)^{t}\gamma_{k}=\gamma_{k}.
\]
Therefore, when $k\in A$ we may choose $1>\xi_{k-1}\ge2^{t}-2$ so
that $D_{k-1}$ satisfies both Assumption-1 and Assumption-2a, but
not Assumption-2b.

\begin{example}
\label{Sec-6-Example:A1-A2a-not-A2b}If $\beta_{k}=\beta$ for all
$k,$ $m_{k}\leq3$ for all $k,$ and $m_{k}=3$ for nearly all $k,$
then $s=\log(3)/\log(1/\beta).$ In particular, $\beta=3^{-t}$ and
(\ref{eq:0-beta}) holds. Since $m_{k}=3$ is possible and $d_{k,j+1}-d_{k,j}>1$
we must have $\frac{1-\beta}{\beta}>2,$ and consequently, $\frac{1}{\beta}>3.$ 

Both (\ref{eq:1-beta}) and (\ref{eq:A-2-alternative}) holds if
and only if $\frac{\log(5)}{\log(3)}\leq t$ and we chose $d_{k,j}$
such that, if $m_{k}=3,$ then $d_{k,0}=0,$ $d_{k,2}=3^{t}-1,$ and
\begin{align*}
2 & \leq d_{k,1}\leq3^{t}-3,\text{ when }m_{k+1}=3\\
2^{t}-1 & \leq d_{k,1}\leq3^{t}-2^{t},\text{ when }m_{k+1}=2
\end{align*}
 and if $m_{k}=2,$ then $0\leq d_{k,0}<d_{k,1}\leq3^{t}-1$ and 
\begin{align*}
2 & \leq d_{k,1}-d_{k,0},\text{ when }m_{k+1}=3\\
2^{t}-1 & \leq d_{k,1}-d_{k,0},\text{ when }m_{k+1}=2.
\end{align*}
If $\frac{\log(5)}{\log(3)}\leq t<\frac{\log(3)}{\log(2)}.$ Then
$2^{t}-1<2,$ hence, if $m_{k+1}=2,$ we can chose $d_{k,j}$ such
that (\ref{eq:1-beta}) and (\ref{eq:A-2-alternative}) and $d_{k,1}-d_{k,0}<2,$
when $m_{k}=2,3$. Alternatively, when $m_{k}=3,$ we can arrange
$d_{k,2}-d_{k,1}<2.$ By Lemma \ref{lem:A-1-Alternative} and Lemma
\ref{lem:A-2-Alternative} this gives examples where Assumption-1
and Assumption-2a holds and Assumption-2b fails. 
\end{example}
\end{example}
\begin{rem}
\label{Sec-6-Rem:QRS} The analysis in Example \ref{Sec-6-Example:A1-A2a-not-A2b}
also leads to examples where the assumptions in Theorem \ref{Sec-2-thm:-QRS}
fails and the assumptions in Theorem \ref{thm:Main} (i.e., $t\geq\log(3)/\log(2)$)
or more generally Theorem \ref{Sec-7-thm:Main-Complete} ($\log(5)/\log(3)\leq t<\log(3)/\log(2)$)
holds. 

Conversely, if we change Example \ref{Sec-6-Example:A1-A2a-not-A2b}
so that $m_{k}=2$ for nearly all $k,$ then (\ref{eq:0-beta}) fails
for all $k$ with $m_{k}=3.$ Consequently, setting $d_{k,0}=0,$
$d_{k,m_{k}-1}=2^{t}-1$ and when $m_{k}=3$ setting $d_{k,1}=\frac{2^{t}-1}{2}$
gives examples where the assumptions of Theorem \ref{Sec-2-thm:-QRS}
holds and the assumptions of Theorem \ref{thm:Main} fail.

\end{rem}
The following example demonstrates that $m_{k}$ can be chosen arbitrarily
large.
\begin{example}
Fix an integer $q\ge2$ and $0<\alpha\le\frac{1}{2}$. For all $k$,
let $m_{k}=q^{k}$ and $\beta_{k}=q^{-k/\alpha}<q^{-2k}$. By choosing
digit sets $D_{k}\subseteq\left[0,\frac{1-\beta_{k}}{\beta_{k}}\right]$
such that $d_{k,0}=0$ and $d_{k,m_{k}-1}=\frac{1-\beta_{k}}{\beta_{k}}$
of $m_{k}$ evenly spaced digits, the set $C$ has dimension $\alpha$
according to equation (\ref{eq:Haudorff-dim-formula}). Note that
$\left(m_{k}\right)^{1/\alpha}\beta_{k}=1$ for all $k$ and 
\[
\mu\left(k\right)\left(b\left(k\right)\right)^{\alpha}=q^{k\left(k+1\right)/2}\left(q^{-k\left(k+1\right)/2\alpha}\right)^{\alpha}=1.
\]

\end{example}

\section{Proof of Theorem \ref{thm:Main}}

It follows from Lemma \ref{lem:A-1-Alternative} and Lemma \ref{lem:A-2-Alternative}
that Theorem \ref{thm:Main} is a special case of the following theorem. 
\begin{thm}
\label{Sec-7-thm:Main-Complete}Let $0<\beta_{k}<1/2$ and
\[
D_{k}=\left\{ d_{k,j}\mid j=0,\ldots,m_{k}-1\right\} \subset\left[0,\frac{1-\beta_{k}}{\beta_{k}}\right],
\]
where $1<d_{k,j+1}-d_{k,j}$ be given. Let $s$ be determined by (\ref{eq:Haudorff-dim-formula})
and let $C$ be as in (\ref{eq:C-definition}). If Assumptions 1 and
2a hold or Assumptions 1 and 2b hold, then $s$--dimensional Hausdorff
measure of $C$ equals
\[
L:=\liminf_{k\to\infty}\mu(k)b(k)^{s},
\]
where $\mu(k)=m_{1}\cdots m_{k}$ and $b(k)=\beta_{1}\cdots\beta_{k}$
are as in (\ref{eq:mu-Def}) and (\ref{Sec-2-eq:b-Def}). \end{thm}
\begin{proof}
Consider a cover $P_{i}$ of $C$ by simple intervals at some stage
$k.$ By Lemma \ref{lem:Basic-vs-Simple-Intervals} 
\[
\sum_{i}\left|P_{i}\right|^{s}\geq\sum_{i}\sum_{j}\left|I_{i,j}\right|^{s}
\]
where $\left(I_{i,j}\right)_{j}$ are the basic intervals contained
in $P_{i}.$ Each $I_{i,j}$ is one of the basic intervals at stage
$k.$ Since the intervals $I_{i,j}$ cover $C,$ the collection $\left(I_{i,j}\right)_{i,j}$
must contain all the basic intervals at stage $k.$ Hence, it follows
from Lemma \ref{lem:Delta-Simple}, that 
\[
H_{\delta}^{s}(C)\geq\sum\left|I\right|^{s},
\]
where the sum is over all basic intervals at stage $k.$ A basic interval
interval at stage $k,$ is a simple interval at stage $k+1,$ hence
Lemma \ref{lem:Basic-vs-Simple-Intervals} tells us that 
\[
\sum\left|I^{(k)}\right|^{s}\geq\sum\left|I^{(k+1)}\right|^{s}
\]
where the first sum is over all basic intervals $I^{(k)}$ at stage
$k$ and the second sum is over all basic intervals $I^{(k+1)}$ at
stage $k+1.$ Consequently, 
\[
H_{\delta}^{s}(C)=\inf_{k\geq\delta^{-1}}\sum\left|I^{(k)}\right|^{s}=\lim_{k\to\infty}\sum\left|I^{(k)}\right|^{s}.
\]
This completes the proof.
\end{proof}


\begin{thebibliography}{99}


\bibitem[1]{Ma86}
Jacques Marion, \emph{Mesure de hausdorff d'un fractal `a similitude interne},
  Ann. Sc. Math. Qu{\'e}bec \textbf{10} (1986), no.~1, 51--81.

\bibitem[2]{Ma87}
Jacques Marion, \emph{Mesures de {H}ausdorff d'ensembles fractals}, Ann. Sc. Math.
  Qu{\'e}bec \textbf{11} (1987), 111--132.


\bibitem[3]{AySt99}
Elizabeth Ayer and Robert~S. Strichartz, \emph{Exact {H}ausdorff measure and
  intervals of maximum density for {C}antor sets}, Trans. Amer. Math. Soc.
  \textbf{351} (1999), no.~9, 3725--3741.

\bibitem[4]{QRS01}
Cheng~Qin Qu, Hui Rao, and Wei~Yi Su, \emph{Hausdorff measure of homogeneous
  {C}antor set}, Acta Math. Sin., English Series \textbf{17} (2001), no.~1,
  15--20.

\bibitem[5]{BoKr11}
Ievgen~V. Bondarenko and Rostyslav~V. Kravchenko, \emph{On {L}ebesgue measure
  of integral self-affine sets}, Discrete Comput. Geom. \textbf{46} (2011),
  389--393.

\bibitem[6]{CMMS04}
Carlos Cabrelli, Franklin Mendivil, Ursula~M. Molter, and Ronald Shonkwiler,
  \emph{On the {H}ausdorff h-measure of {C}antor sets}, Pacific J. Math.
  \textbf{217} (2004), no.~1, 45--59.

\bibitem[7]{FRJ96}
D.~J. Feng, H.~Rao, and J.~Wu, \emph{The net measure properties of one
  dimensional homogeneous {C}antor set and its applications}, Progr. Nat. Sci.,
  \textbf{6} (1996), no.~6, 673--678.

\bibitem[8]{GMS07}
Ignacio Garcia, Ursula Molter, and Roberto Scotto, \emph{Dimension functions of
  {C}antor sets}, Proc. Amer. Math. Soc. \textbf{135} (2007), 3151--3161.

\bibitem[9]{HMZ13}
Kathryn~E. Hare, Franklin Mendivil, and Leandro Zuberman, \emph{The sizes of
  rearrangements of {C}antor sets}, Can. Math. Bull. \textbf{56} (2013),
  354--365.
  

\bibitem[10]{PePh13}
Steen Pedersen and Jason~D. Phillips, \emph{On intersections of {C}antor sets: {H}ausdorff measure},
  Opuscula Math. \textbf{33} (2013), no.~3, 575--598.  
    

\bibitem[11]{PaTa87}
Jacob Palis and Floris Takens, \emph{Hyperbolicity and the creation of
  homoclinic orbits}, Ann. Math. \textbf{125} (1987), 337--374.
  
      
\bibitem[12]{DaHu95}
G.~J. Davis and T-Y Hu, \emph{On the structure of the intersection of two
  middle thirds {C}antor sets}, Publ. Math. \textbf{39} (1995), 43--60.
  

\bibitem[13]{MR0354562}
Harry Furstenberg, \emph{Intersections of {C}antor sets and transversality of
  semigroups}, Problems in analysis ({S}ympos. {S}alomon {B}ochner, {P}rinceton
  {U}niv., {P}rinceton, {N}.{J}., 1969), Princeton Univ. Press, Princeton,
  N.J., 1970, pp.~41--59.

\bibitem[14]{MR1112813}
R.~F. Williams, \emph{How big is the intersection of two thick {C}antor sets?},
  Continuum theory and dynamical systems ({A}rcata, {CA}, 1989), Contemp.
  Math., vol. 117, Amer. Math. Soc., Providence, RI, 1991, pp.~163--175.
  
\bibitem[15]{MR1106988}
Roger Kraft, \emph{Intersections of thick {C}antor sets}, Mem. Amer. Math. Soc.
  \textbf{97} (1992), no.~468, vi+119.

\bibitem[16]{MR1653359}
Roger~L. Kraft, \emph{Random intersections of thick {C}antor sets}, Trans.
  Amer. Math. Soc. \textbf{352} (2000), no.~3, 1315--1328.

\bibitem[17]{MR1491873}
Yuval Peres and Boris Solomyak, \emph{Self-similar measures and intersections
  of {C}antor sets}, Trans. Amer. Math. Soc. \textbf{350} (1998), no.~10,
  4065--4087.



\bibitem[18]{MR2534911}
Shu-Juan Duan, Dan Liu, and Tai-Man Tang, \emph{A planar integral self-affine
  tile with {C}antor set intersections with its neighbors}, Integers \textbf{9}
  (2009), A21, 227--237.
  


\bibitem[19]{MR2810854}
Carlos~Gustavo Moreira, \emph{There are no {$C^1$}-stable intersections of
  regular {C}antor sets}, Acta Math. \textbf{206} (2011), no.~2, 311--323.  


\bibitem[20]{Fal85}
Kenneth.~J. Falconer, \emph{The geometry of fractal sets}, Cambridge University
  Press, Cambridge, 1985.



\bibitem[21]{PePh11}
Steen Pedersen and Jason~D. Phillips, \emph{Intersections of certain deleted
  digits sets}, Fractals \textbf{20} (2012), 105--115.




\end{thebibliography}
\end{document}